\documentclass[10pt,twocolumn,twoside]{ieeeconf}
\IEEEoverridecommandlockouts    
\usepackage[margin=0.75in]{geometry}

\newtheorem{theorem}{Theorem}
\newtheorem{corollary}{Corollary}
\newtheorem{lemma}{Lemma}
\newtheorem{remark}{Remark}
\newtheorem{definition}{Definition}
\newtheorem{proposition}{Proposition}
\newtheorem{example}{Example}

\usepackage{dsfont}
\usepackage[utf8]{inputenc}
\usepackage{amssymb}
\usepackage{amsmath}
\usepackage{color}
\usepackage{tabularx}
\usepackage{svg}
\usepackage{overpic}
\usepackage{subfigure}
\usepackage{bbm}
\usepackage{mathtools}
\usepackage{comment}
\usepackage{cite}
\let\labelindent\relax
\usepackage[shortlabels]{enumitem}
\usepackage{graphicx}
\usepackage{epstopdf}
\usepackage{derivative}
\usepackage{soul}
\usepackage{tikz-network}

\newcommand{\E}{\mathcal{E}}
\newcommand{\V}{\mathcal{V}}
\newcommand{\G}{\mathcal{G}}
\newcommand{\Z}{\mathcal{Z}}

\newcommand{\phil}[1]{\textcolor{cyan}{#1}}

\title{Optimal Ordering Policies for Multi-Echelon Supply Networks}
\author{Jos\'e I. Caiza, Ian Walter, Jitesh H. Panchal, Junjie Qin, and Philip E. Par\'e*\thanks{*Jos\'e I. Caiza, Ian Walter, Junjie Qin, and Philip E.~Par\'e are with the Elmore Family School of Electrical and Computer Engineering at Purdue University. Jitesh H. Panchal is with the School of Mechanical Engineering at Purdue University. Emails:~jcaiza@purdue.edu, walteri@purdue.edu, panchal@purdue.edu, jq@purdue.edu, philpare@purdue.edu. This work was supported in part 
   by the National Science Foundation, grant
   NSF-ECCS \#2032258.}}

\begin{document}

\maketitle

\begin{abstract}
    In this paper, we formulate
    an optimal ordering policy 
    as 
    a stochastic control 
    problem where each firm decides the amount of input goods to order from their upstream suppliers based on the current inventory level of its output good. 
    For this purpose, we provide a closed-form solution for the optimal request of the raw materials for given a fixed production policy.
    We implement the proposed policy
    on a 15-firm acyclic network based on a real product supply chain.
    We first simulate ideal demand situations, 
    and then we implement demand-side shocks (i.e., demand levels outside of those considered in the policy formulation) and supply-side shocks (i.e., halts in production for some suppliers) 
    to evaluate the robustness of the proposed policies.
\end{abstract}

\section{Introduction}
    In a supply network (SN), firms maximize their profits by reducing their  holding costs and stockouts (shortages) to meet 
    demand.
    This inventory management problem has been a topic of interest for researchers for over a century, with the introduction of the economic order quantity model in 1913~\cite{harris1913}.
    The body of research has grown considerably in the nearly eleven decades since, with researchers expanding the problem from a single firm to looking at entire supply chains composed of multiple stages, referred to as multi-echelon supply chains~\cite{clark1960optimal}.
    There has been a wave of literature reviews and summaries on this area of research over the past decade, either highlighting the state of the art in general~\cite{bushuev2015review,firoozi2018multi} or highlighting specific subsections of research, such as work using stochastic demand~\cite{DEKOK2018955} or focusing on the various types of control strategies used~\cite{sun2020review}.
    The work done on these multi-echelon supply chains can be decomposed further by the optimization method used~\cite{firoozi2018multi,sun2020review}, e.g.,
    heuristics~\cite{YEONGJOONYOO1997729,HNAIEN2017223}, genetic algorithms~\cite{ZHOU20132039}, mixed integer programming~\cite{firouz2017integrated,firoozi2018multi}, or dynamic programming~\cite{graves2000optimizing,graves2008strategic,parker04twoechelon}; the type of demand used~\cite{DEKOK2018955}, e.g., deterministic, stationary stochastic, non-stationary stochastic; and the complexity of the network structure~\cite{firoozi2018multi}, e.g., serial, two-echelon, spanning tree.
    
    
    
    Much of the recent literature on supply chain management has been focused on safety-stock placement, or where in the network excess stock should be stored in order to meet targeted delivery times for uncertain consumer demand~\cite{ERUGUZ2016110,graves2008strategic,graves2016strategic}.
    Once optimal stock placement locations are determined, firms in the network can operate using simple base-stock policies for inventory management, which reduces the need for communication spanning the majority of the supply chain.
    On the other hand, the strategies to solve multi-stage stochastic problems have been developed so the policies can be computed without optimization and to get a better insight of the conditions ruling the optimal solutions~\cite{7040021}. This paper considers networks without any strategic placement of inventory, and given such a network attempts to determine the optimal ordering policy for each firm.

    In this work, we formulate a multi-stage stochastic control problem to optimize a single firm's supply ordering policy, given a known stationary stochastic demand distribution. To that end, we rigorously 
    prove the existence of a threshold policy by using Bellman's recursion.  
    We then apply the proposed policies to each firm in a network in a distributed fashion
    by propagating the demand distribution from the distributor firms through their decision-making processes and approximating the upstream demands via Monte Carlo simulations. 
    Our proposed solution can be applied to any general multi-echelon networks
    under the assumption that supplier firms do not require inputs from downstream firms.
    In our simulations we compare the effect this distributed optimization ordering policy framework has on the individual firms' profits as well the overall social welfare of the network.

    The paper is organized as follows.
    We first introduce in Section~\ref{sec:model} the characteristics of our proposed SN model.
    In Section~\ref{sec:optimization} we formulate the stochastic control problem for multistage requests of each firm, and derive closed form solution to the optimal requests and the cost-to-go function.
    In Section~\ref{sec:simulations} we 
    illustrate the effectiveness of the proposed framework via simulations. 
    Lastly, in Section~\ref{sec:conclusion} we conclude and highlight future directions\phil{.}


    
    

\section{Model Definition} \label{sec:model}


In this section we first formally define the structure of our SN and introduce the characteristics of firms and goods. We then describe how demand is modeled for the two different types of firms.
Lastly, we introduce 
the inventory dynamics to facilitate the formulation of 
the optimization problem.

\subsection{SN Model} 
We begin by defining our production network as the graph~$\G=(\V,\E)$. 
There are $n$ different firms in the network, denoted by the set of vertices $\V = \{v_1,\dots,v_n\}$. 
Each firm $v_i$ is located in one of $L$ echelons, denoted $l_i\in[1,\dotsc,L]$. 
For example if firm $v_i$ is in the most downstream (rightmost) echelon, then $l_i=1$ as seen in Figure~\ref{fig:notation_fig}. firms in echelon $l_i=1$ are termed distributors and have no firms as out-neighbors, and firms in echelons $l_i>1$ are termed suppliers. 
The directed edge set $\E$ contains all supply chain connections between firms in the graph.
If $(v_i,v_j)\in \E$, 
there is a flow of material from $v_i$ to $v_j$ for at least some time-step~$k$.
We also define the set of node $v_i$'s out-neighbors 
as~$\mathcal{N}_i^-$ and the set of  in-neighbors as~$\mathcal{N}_i^+$.

There are $m=n$ different goods in the network, and this number is fixed over all time-steps. 
The set of all good types is given by $\Z = \{z_1,\dotsc,z_m\}$. 
Each 
$z_i$ has a bill of materials $B_i\in \mathbb R^m$, which is the vector containing the quantities of intermediary goods or inputs required for the manufacturing of one unit of $z_i$. $B_i$ has a length of $m$, and does not change with time. 
Several examples of $B_i$ are given in the caption of Figure~$\ref{fig:notation_fig}$.
We assume each firm~$v_i$ outputs one and only one type of good, 
and because $m=n$ there is exactly one firm manufacturing each good with the same index. In other words, $v_i$ is the only firm that manufactures good $z_i$.

\begin{figure} 
    \centering
    \begin{tikzpicture}[every label/.append style={text=red, font=\large}]
        \Vertex[y=1.5,size=1,label=$v_1$,fontsize=\large]{1} \Vertex[y=-1.5,size=1,label=$v_2$,fontsize=\large]{2} \Vertex[x=4,size=1,label=$v_3$,fontsize=\large]{3} \Vertex[x=6,shape=coordinate]{demand}
        \Edge[Direct,lw=3pt,Math,label=u^k_{31},fontsize=\normalsize](1)(3)
        \Edge[Direct,lw=3pt,Math,label=u^k_{32},fontsize=\normalsize](2)(3)
        \Edge[Direct,lw=3pt](3)(demand)
        \Text[y=1,position=below]{$z_1$}
        \Text[y=-2,position=below]{$z_2$}
        \Text[x=4,y=-.5,position=below]{$z_3$}
        \Text[x=5,position=above right, distance=1mm]{$\omega_3^k$}
        \Text[x=4,y=-3]{$l_3=1$}
        \Text[y=-3]{$l_1=l_2=2$}
        \draw[dashed] (3.1,-3) -- (3.1,2);
        \draw[dashed] (4.9,-3) -- (4.9,2);
        \draw[dashed] (-.9,-3) -- (-.9,2);
        \draw[dashed] (.9,-3) -- (.9,2);
    \end{tikzpicture}
    \caption{In this example graph used to illustrate the application of our notation, firm $v_3$ is in echelon $l_3=1$ and is a distributor of product $z_3$ and $v_1,v_2$ are both in echelon $l_1=l_2=2$ and are suppliers of goods $z_1$ and $z_2$, respectively. The bill of materials for goods $z_1$ and $z_2$ are both~$B_1=B_2=[0,0,0]^T$ as we assume their production requires no other raw materials, while for $z_3$ the bill of materials is $B_3 = \left[2,1,0\right]^T$. The random demand for good $z_3$ at time-step $k$ is $w_3^k$, and the demand seen by firms $v_1$ and $v_2$ are $u_{31}^k$ $u_{32}^k$, which are the orders placed by $v_3$ for each good.}
    \label{fig:notation_fig}
\end{figure}
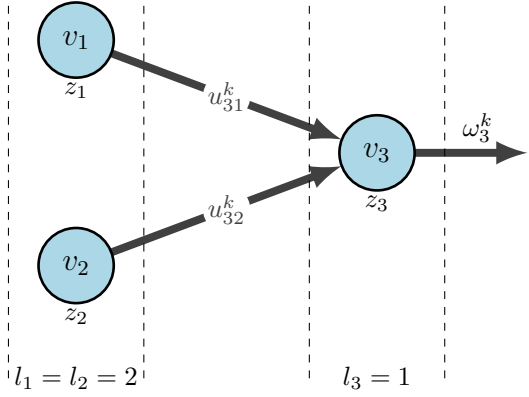

Each firm $v_i$ has an inventory policy for both intermediate goods (inputs) and products (outputs). For each good $z_r$ that is either $z_i$ or has a non-zero value in $B_i$, firm $v_i$ keeps an on-site inventory with target minimum and maximum values.
The way a firm manages inventory for each good is referred to as their inventory policy, 
and in this paper each firm $v_i$ needs to maintain its inventory $x_{ir}^k$ between the given lower bound $x^L_{ir}$ and upper bound $x^U_{ir}$. That is, $x^L_{ir}\leq x^k_{ir}\leq x^U_{ir}$.

\subsection{Demand Modeling--Distributors}
Let $\mathcal Z^D$ be the set of goods produced by the distributors. We capture the demand for all goods $r \in \mathcal Z^D$ in the set of demands $\mathcal{W}^k=\{\omega^k_r\}_{r\in \mathcal Z^D}$. 
At each time-step $k$, $\Omega^k_r$ is a random variable associated with a probability mass function (pmf) 
$p_{\Omega_r}(\omega_r^k)$
for good $z_r$, such that $p_{\Omega_r}(\omega_r^k)=\mathbb{P}(\Omega_r=\omega_r^k)$.
Because $m=n$, all demand for a product $z_i$ is directed to the $v_i$.
We also assume the demand faced by any distributor is independent of all other distributor demands and it has a support bounded in the interval $\mathcal{I}_i^k=[d_i^k,D_i^k]$.

\subsection{Inventory Dynamics}\label{sec:inv_dynamics}
In addition to deciding how much of each product to order at time $k$, firm $v_i$ also makes a decision on how much good $z_i$ to produce. 
We define the variable~$N_i^k$ to be the amount of good $z_i$ firm $v_i$ plans to  request at time $k$. 
We define~$y_i^k = \left[y_{i1}^k \ y_{i2}^k \ \cdots \  y_{im}^k\right]^T$ as the vector of on-hand stock of each good $z_r$ at firm $v_i$, defined as $y_{ir}^k = x_{ir}^k +  N^k_{i}B_{ir}$ 
with $x_{ir}^k$ being the inventory level of good $r$ at the start of time-step $k$.
The production function~$\pi$ could depend on additional variables directly, such as demand.
However, in this work each firm~$v_i$ uses a
rudimentary production policy such that 
the maximum amount of $z_i$ will be produced 
using the on-hand stock $y_{ir}^k$  under the constraints of $B_i$. 

For firm $v_i$, the inventory dynamics 
of any good $z_r$ at time $k$ are determined by the current inventory levels,
the amount of good received $N^k_{i}B_{ir}$, and the random demand~$\omega^k_i$. 
We describe these dynamics as
\begin{equation}\label{eq:init_inv_dynamics}
    x_{ir}^{k+1} = \begin{cases} 
      x_{ii}^k + N_i^k - \omega_i^k  & r = i \\
      y_{ir}^k - \pi(y_{i}^k)B_{ir} & r \neq i ,
      \end{cases}
\end{equation} 
where the first equation is for the output good of firm $v_i$, the second equation holds for input goods, 
and $B_{ir}$ is the amount of $z_r$ used to produce $z_i$, i.e., $B_{ir}$ is the $r$th entry of $B_i$.
Each firm $v_i$ also has a number of costs/penalties relating to the good produced $z_i$ that subtract from profits: 
the production cost per unit $c_i$, the penalty for having a shortage $s_i$, and the holding cost $h_i$. 



\subsection{Demand Modeling--Suppliers}
The amount of a good $z_j$ flowing from firm $v_j$ to $v_i$  
at time $k$ is the amount of $z_j$ ordered by $v_i$, and is 
given by $N_j^k B_{ji}$.
We assume no time lag between order placement and fulfillment, i.e., an order is placed at the beginning of time-step $k$ and is delivered that same time-step.
The amount ordered is determined by the inventory control policy used by firm $v_i$, and 
$N_i^k$
is also referred to as the control action.
%
%

For goods output by a supplier firm $v_i$, that is, for $i\in \mathcal{W}^k\setminus\mathcal Z^D$, 
the demand at each time-step is the sum of all orders placed for good $z_i$
, that is, for $v_i$ such that $l_i>1$,
\begin{equation}\label{eq:Udef}
    \omega^k_i
    = \sum_{j\in\mathcal{N}_i^-}^{} 
    N_j^k B_{ji},
\end{equation}
where, recall, $\mathcal{N}_i^-$ is the set of all out-neighbors of firm $v_i$.


\subsection{Problem Formulation} \label{subsec:problem}
We take the perspective of a production network manager trying to find the optimal ordering strategy to manage inventory levels at each firm.
In other words, we want to find $N_i^k$ at each time-step $k$ that minimizes a
cost function dependent on the parameters $c_i$, $s_i$, and $h_i$ at each time-step.
Given that we find an optimal strategy, we then want to understand how incorporating a production decision $\pi(y_i^k)$ impacts the response of our system.
Lastly, we aim to understand how 
shocks (e.g., production outages $N_i^k = 0$ and shortages $x_{ii}^k+N_i^k < \omega^k_i$) 
impact the behavior of a given network and check the robustness of the production-inventory policy.

\section{Optimal Request}\label{sec:optimization}

We consider the profit-availability goal of the supply chain market, in which a firm in the network needs to supply goods from a sequence of downstream firms to meet a random demand. In this section we develop the scope of the decision problem 
to look for optimal inventory policies for output goods $r =i$ at each firm.




Taking into account the aspects mentioned in Section~\ref{subsec:problem} and the dynamics defined in \eqref{eq:init_inv_dynamics}, 
the decision problem of each individual node can be addressed by a stochastic control program: 
\begin{subequations} \label{eq:cost_function_N}
    \begin{align}
        &\text{minimize} \ \ \mathbb{E}\Bigg\{ \sum_{k=0}^{T-1} \bigg( c_i N_i^k +  f_i(x^k_{ii}+N_i^k-\omega_i^k)\bigg)  \Bigg\} \\
        & \ \ \   \text{s.t.} \ \ \ \ \ \ x_{ii}^{k+1} = x_{ii}^k + N_i^k - \omega_i^k,  \ \ \ \ \ \ \ \ \ \ \ \ \ \ \ \ \ \label{eq:DP_Algorithm_3}\\
        &\ \ \ \ \ \ \ \ \ \ \ \  x^L_{ii} \leq x_{ii}^k \leq x^U_{ii},\label{eq:DP_Algorithm_4}\\
        &\ \ \ \ \ \ \ \ \ \ \ \  x^L_{ii} + D_i^k \leq x_{ii}^k + N_i^k \leq x^U_{ii},\label{eq:DP_Algorithm_5}\\
        &\ \ \ \ \ \ \ \ \ \ \ \  N_i^k \geq 0,
    \end{align}
\end{subequations}
where $f_i(x^T_{ii}) = 0$, and we aim to find a control policy at each stage $k = 0,\dots,T-1$, that maps the current inventory level $x_{ii}^k$ to the optimal request $N_i^{k}$, where the price of producing $N_i^k$ goods is represented by $c_i>0$. Here $f_i$ is a penalty function tracking the systematic risk at each firm due to an unmet demand or having too many output goods stored. 
An example of common use in the stochastic inventory control literature is the holding/shortage penalty
{\small
\begin{align*}
    f_i(x^k_{ii}+N_i^k-\omega_i^k) &= s_i\big(\omega_i^k-x^k_{ii}-N_i^k\big)_+ + h_i\big(x^k_{ii}+N_i^k-\omega_i^k\big)_+,
\end{align*}}
where $(u)_+ = \text{max}(u,0)$. The penalty characterizes the cost of having a shortage on the inventory level $\big(\omega_i^k-x^k_{ii}-N_i^k\big)_+$ at a price $s_i$ and having a holding $\big(x^k_{ii}+N_i^k-\omega_i^k\big)_+$ at a price $h_i$, where $s_i,h_i>0$.  

Observing the additive structure of the updates in $x_i^k$ and $N_i^k$, we can rewrite \eqref{eq:cost_function_N} in the following form:
\begin{subequations} \label{eq:cost_function_a}
    \begin{alignat}{4}  
        &\text{minimize} \ \ \mathbb{E}\Bigg\{ \sum_{k=0}^{T-1} \bigg( c_i N_i^k +  f_i(a^k_{ii}-\omega_i^k)\bigg)  \Bigg\} \\
        &\ \ \ \   \text{s.t.} \ \ \ \ \ \ x_{ii}^{k+1} = y_{ii}^k - \omega_i^k,  \ \ \ \ \ \ \ \ \ \ \ \ \ \ \ \ \ \\
        &\ \ \ \ \ \ \ \ \ \ \ \ \ \ x^L_{ii} + D_i^k \leq y_{ii}^k \leq x^U_{ii},\\
        &\ \ \ \ \ \ \ \ \ \ \ \ \ \ y_{ii}^k \geq x_{ii}^k,
    \end{alignat} 
\end{subequations}
where $y_{ii}^k = x_{ii}^k + N_i^k, k = 1,\cdots,T-1$, is the on-hand stock of output goods after producing $N_i^k$ units and the expectation is taken over $\omega_i^k$.
Note that the on-hand stock $y_{ii}^k$ is constrained to remain above the lower bound of the current stock $x_{ii}^k$ plus the maximum possible demand $D_i^k$ of an output good. 
This constraint will allow the optimization of the request $N_i^k$ to meet the constraints imposed on $x_{ii}^k$ at any stage $k$.
%
The stochastic control problem defined above can be solved with the Bellman's recursion. To this end, we define the cost-to-go function as follows: 
\begin{subequations}
    \begin{alignat}{3}  
        J_i^{T}(x^T_{ii}) =& \thinspace 0, \label{eq:DP_Algorithm_1}\\
        J_i^{k}(x^k_{ii}) = &\underset{y_{ii}^k\geq x_{ii}^k}{\text{min}} Q_i^k (y_{ii}^k) - c_i x_{ii}^k \label{eq:DP_Algorithm_2}\\
        &\ \ \text{s.t.} \ \ \ \underline{a}_{ii} \leq y_{ii}^k \leq \overline{a}_{ii}, \label{eq:constraint_on-hand}
    \end{alignat}
\end{subequations}
%
%
\normalsize
\noindent
for $k=0,\dots,T$, where $\underline{a}_{ii} = x_{ii}^L + D_i^k$, $\overline{a}_{ii} = x_{ii}^U$. The objective of \eqref{eq:DP_Algorithm_2} is defined as the state-action cost-to-go function, such that, for stage $k=0,\dots,T-1$, we have 
\vspace{-0.5ex}
\begin{equation}\label{eq:state_cost}
    Q_i(y_{ii}^k) = c_iy_{ii}^k +  \mathbb{E}\Big\{f_i(y_{ii}^k-\omega_i^k) + J_i^{k+1}\big(y_{ii}^k-\omega_i^k\big)\Big\}.
\end{equation}
At each stage $k$, the cost-to-go function takes into account the minimum expected cost from the current stage with the given on-hand stock $y_{ii}^k$ at the end of the decision process. Thus, for any stage $k$ the optimal on-hand stock can be solved from 
\vspace{-0.5ex}
\begin{equation} \label{eq:general_policy}
    y_{ii}^{k,\star} = \underset{y_{ii}^k\geq x_{ii}^k}{\text{argmin}} \  Q_i(y_{ii}^k),
\end{equation}
where $y_{ii}^{k,\star}$ is subjected to \eqref{eq:constraint_on-hand}.

Under a numerical approach, 
$\eqref{eq:DP_Algorithm_2}$, $\eqref{eq:state_cost}$, and $\eqref{eq:general_policy}$ can be seen as a definition of the optimal conditions for $\eqref{eq:cost_function_a}$. However, the Bellman's recursion does not lead to a polynomial time algorithm for solving the stochastic control problem as both the state space and the action space are continuous. To circumvent this challenge, we proceed to analytically solve the stochastic control problem by characterizing the structure of the optimal control policy.
\begin{proposition}\label{prop:convex_a}
The cost-to-go function $J_{i}^k(x_{ii}^k)$, for all $k=0,\cdots,T-1$ is convex in $x_{ii}^k$ given that the state cost-to-go function $Q_i^k$ is convex in the on-hand stock $y_{ii}^k$.
\end{proposition}
\begin{proof}
Note that $J_i^T$ is trivially convex by its definition. Now, assume that $J_i^{k+1}$ is convex in $x_{ii}^k$. Since, the expectation preserves the convexity, $Q(y_{ii}^k)-c_ix_{ii}^k$ is jointly-convex in $(y_{ii}^k,x_{ii}^k)$, where $C=\big\{(y_{ii}^k,x_{ii}^k): y_{ii}^k \geq x_{ii}^k\big\}$ is a convex set on $(y_{ii}^k,x_{ii}^k)$. Thus, by Proposition \ref{prop:convex_general} (see the Appendix), we have that $J_i^k(x_{ii}^k)$ is convex in $x_{ii}^k$. 
By induction, we can conclude that for all $~k = 0,\dots,T-1$, $J_i^k(x_{ii}^k)$ is convex in $x_{ii}^k$.
\end{proof}
A direct conclusion of Proposition $\ref{prop:convex_a}$ is the following form of the optimal request policy $N_i^{k,\star}(x_{ii}^k)$:
\begin{lemma} \label{lem:dp}
For each stage $k=0,\dots,T-1$, there exists a threshold $M_i^k$ which is independent of the current inventory level $x_{ii}^k$ and is
a function of the parameters $c_i$, $s_i$, $h_i$, such that the optimal request policy has the form 
\begin{equation}\label{eq:threshold}
    N_i^{k,*} = \Big[(M_i^k - x_{ii}^k)_+\Big]^{\overline{N}_i^k(x_{ii}^k)}_{\underline{N}_i^k(x_{ii}^k)},
\end{equation}

where $[x]^\alpha_\beta := \min\{\alpha, \max\{x, \beta\}\}$ and 
\begin{align*}
    \overline{N_i}^k(x_{ii}^k)&:=x_{ii}^U - x_{ii}^k,\\
    \underline{N_i}^k(x_{ii}^k)&:= x_{ii}^L + D_i^k - x_{ii}^k.
\end{align*}
\vspace{-3ex}
\end{lemma}
\begin{proof}
Take into account the Bellman's recursion at stage $k$ as in $\eqref{eq:DP_Algorithm_2}$, in which the optimization is solved for a fixed $x_{ii}^k$, where $Q_i(y_{ii}^k)$ is convex by Proposition $\ref{prop:convex_a}$. Let $M_i^k= \text{argmin}_{y_{ii}^k\in \mathbb{R}}Q_i(y_{ii}^k)$, thus the constrained minimizer of $y_{ii}^k$ in $\eqref{eq:DP_Algorithm_2}$ is of the form $y_{ii}^{k,\star} = \text{max}(M_i^k,x_{ii}^k)$. Note that, it is trivial that our claim holds when $M_i^k \geq x_{ii}^k$. Otherwise, we show that $y_{ii}^{k,\star} = x_{ii}^k$ by contradiction. Assume that $y_{ii}^{k,\star}=\tilde{a}_{ii}^{k,\star}\neq x_{ii}^k$ such that $Q_i(\tilde{a}_{ii}^{k,\star})<Q_i(x_{ii}^k)$. Thus, we can say there exists an $\alpha \in (0,1)$, such that $x_{ii}^k = \alpha \tilde{a}_{ii}^{k,\star} + (1-\alpha)M_i^{k}$. By convexity of $Q_i$ in $y_{ii}^k$, we have 
\begin{align*}
    Q_i(\tilde{a}_{ii}^{k,\star})<Q_i(x_{ii}^k) &= Q_i(\alpha \tilde{a}_{ii}^{k,\star} + (1-\alpha)M_i^{k})\\
    &\leq \alpha Q_i( \tilde{a}_{ii}^{k,\star}) + (1-\alpha)Q_i(M_i^k),
\end{align*}
with the first inequality due to assumption of the existence of $\tilde{a}_{ii}^{k,\star}$ as the minimizer of $Q_i(y_{ii}^k)$ instead of $x_{ii}^k$. Therefore, $Q_i(\tilde{a}_{ii}^{k,\star})<Q_i(M_i^k)$,which is a contradiction since $M_i^k$ is the minimizer of the unconstrained problem. Therefore, $y_{ii}^{k,\star}=\text{max}(M_i^k,x_{ii}^k)$ solves $\eqref{eq:DP_Algorithm_2}$. Now, in order to assure that $y_{ii}^{k,\star}$ is in the feasible set of $y_{ii}^k$ for any $k=1,\cdots,T-1$, we will project $\text{max}(M_i^k,x_{ii}^k)$ into $\eqref{eq:constraint_on-hand}$, which yields the following characterization of the optimal policy respect to $y_{ii}^k$: 
\begin{align*}
    y_{ii}^{k,\star} = \Big[\text{max}\big(M_i^k,x_{ii}^k\big)\Big]^{x_{ii}^U}_{x_{ii}^L+D_i^k},
\end{align*}
where $[x]^\alpha_\beta:=\text{min}\{\alpha,\text{max}\{x,\beta\}\}$. Recalling, the relationship between $y_{ii}^k$ and $N_i^k$, $y_{ii}^k = x_{ii}^k + N_i^{k}$, we conclude that
\begin{align*}
    N_i^{k,\star}\big(x_{ii}^k\big) = \Big[\big(M_i^k-x_{ii}^k\big)_+ \Big]^{x_{ii}^U-x_{ii}^k}_{x_{ii}^L+D^k-x_{ii}^k},
\end{align*}
where $\overline{N}_i^k(x_{ii}^k):= x_{ii}^U-x_{ii}^k$ and $\underline{N}_i^k(x_{ii}^k):=x_{ii}^L+D_i^k-x_{ii}^k$.
\end{proof}
So far, we have provided the formal proofs on the convexity of the cost-to-go function and the characterization of the optimal policy. Now, we proceed to provide a closed-form characterization of the thresholds. To that end, we define the following mapping functions that yields from backward recursion. Let $c_i,s_i,h_i>0$. Then, we have:
\begin{subequations}\label{eq:F}
    \begin{align}
        F_i^{T-1}(y_{ii}) &= \mathbb{E}\big[f^\prime(y_{ii}-\omega_i^{T-1})\big],\\
        \begin{split}
        F_i^k (y_{ii}) &= \mathbb{E} \bigg\{ f^\prime(y_{ii}-\omega_i^k)-c_i \mathds{1}\big(F_i^{k+1}(y_{ii}-\omega_i^k)\leq -c_i\big) \\
    &\quad +F_i^{k+1}(y_{ii}-\omega_i^k)\mathds{1}\big(F_i^{k+1}(y_{ii}-\omega_i^k)>-c_i\big)\bigg\},
    \end{split}
    \end{align}
\end{subequations}
for $k=0,\cdots,T-1$, where $f^\prime(a_{ii}-\omega_i^{k})$ is the derivative, wherever it is defined, of the penalty function $f$ with respect to $y_{ii}$. Let $F_i^{k,\text{min}} = \inf_{y_{ii}\in \mathbb{R}}F_i^k(y_{ii})$ and $F_i^{k,\text{max}} = \sup_{y_{ii}\in \mathbb{R}}F_i^k(y_{ii})$. Define, for $k=0,\dots,T-1$,
\begin{equation}\label{eq:threshold_definition}
    \tilde{M}_i^k = 
    \begin{cases}
    {F_i^k}^{-1}(-c_i) & \text{if $F_i^{k,\text{min}}<-c_i<F_i^{k,\text{max}}$},\\
    x^L + D_i^k & \text{if $-c_i \leq F_i^{k,\text{min}}$},\\
    x^U & \text{if $-c_i \geq F_i^{k,\text{max}}$},
    \end{cases}
\end{equation}
where ${F_i^k}^{-1}(-c_i)=\inf\{y_{ii}:F_i^k(y_{ii}) \geq -c_i\}$, and
\begin{subequations}\label{eq:G}
    \begin{align}
        G_i^{T-1}(y_{ii}) &= \mathbb{E}[f(y_{ii}-\omega_i^{T-1})],\\
        \begin{split}
            G_i^k(y_{ii}) &= \mathbb{E}\bigg\{f(y_{ii}-\omega_i^k)+\Big[c_i(M_i^{k+1}-y_{ii}+\omega_i^k)\\
            &\qquad +G_i^{k+1}(M_i^{k+1})\Big] \mathds{1}(F_i^{k+1}(y_{ii}-\omega_i^k)\leq-c_i)\\
            &\quad + G_i^{k+1}(y_{ii}-\omega_i^k) \mathds{1}(F_i^{k+1}(y_{ii}-\omega_i^k)>-c_i)\bigg\},
        \end{split}
    \end{align}
\end{subequations}
for $k=0,\dots,T-1$. 
\begin{lemma}\label{lem:T_G}
The following claims are true for all stage $k=0,\dots,T-1$ in the sequence of functions $\big\{ F_i^k\big\}$ and $\big\{G_i^k\big\}$:
    \begin{enumerate}[label=(\alph*)]
        \item $\tilde{M}_i^k$ is well-defined;
        \item $G_i^k(y_{ii})$ is convex and differentiable;
        \item ${G_i^k}^\prime(y_{ii})=F_i^k(y_{ii})$;
        \item $F_i^k(y_{ii})$ is nondecreasing, so $F_i^{k,\text{min}}=F_i^k(x_{ii}^L+D_i^k)$ and $F_i^{k,\text{max}}=F_i^k(x_{ii}^U)$.
    \end{enumerate}
\end{lemma}
\begin{proof}
    See the Appendix.
\end{proof}
Note that, the sequences of mapping functions $\{F_i^k\}$ and $\{G_i^k\}$ are evaluated without any optimization algorithm. Given the density of the demand distributions $\omega_i^k$ at each firm, the computation of these functions can be done by numerical integration or Monte-Carlo simulation. We now define the optimal request and cost-to-go functions in terms of the functions defined in Lemma $\ref{lem:T_G}$.
\begin{theorem} \label{theo:J}
For $k=0,\dots,T-1$, the threshold in \eqref{eq:threshold} is defined as it is stated in \eqref{eq:threshold_definition}, having that the optimal request is
\begin{equation}\label{eq:threshold_policy}
    N_i^{k,\star} = \Big[\big(M_i^k-x_{ii}^k\big)_+ \Big]^{x_{ii}^U-x_{ii}^k}_{x_{ii}^L+D^k-x_{ii}^k} = (\tilde{M}_i^k-x_{ii}^k)_+,
\end{equation}
and the cost-to-function is 
\begin{equation} \label{eq:general_cost_to_go}
        J_i^k(x_{ii}^{k}) = \begin{cases} 
      c_i(\tilde{M}_i^{k}-x_{ii}^{k}) + G_i(\tilde{M}_i^{k})  &  \tilde{M}_i^{k}\geq x_{ii}^{k}  \\
      G_i(x_{ii}^{k}) &  \tilde{M}_i^{k}<x_{ii}^{k}.
      \end{cases}
    \end{equation}
\end{theorem}
\begin{proof}
    See the Appendix.
\end{proof}
\begin{remark}[Optimal Requests in Multi-Echelon SN]
We have provided an analytical result to prove the existence of an optimal request policy of each individual node in the network and how it can be constructed using backward recursion. Under a network perspective, we are interested in how the policies of different nodes are connected to each other. To that end, we proceed to implement simulations based on the fact that a threshold policy is achievable for each firm and assess the network behavior when scenarios not considered in the problem formulation arise. 
\end{remark}

\section{Simulations} \label{sec:simulations}
In this section we implement the optimal ordering policies  
in code on the multi-echelon supply chain network visualized in Figure~\ref{fig:simulation_network}.
The network structure considered is loosely based on a physical supply chain for the production of a Kodak digital camera~\cite{graves2000optimizing}.
We first discuss how these simulations are implemented in code, then show the results of these simulations under ideal and non-ideal conditions to naively evaluate the policy robustness.
We explore the strengths and weaknesses of the policy and then explore its robustness by imposing supplier shortages and demand shifts.
One approximation used in our simulations is that we constraint $\pi(y_i^k)$ to output only integer values.

\begin{figure}
    \centering
    \includegraphics[width=.95\columnwidth]{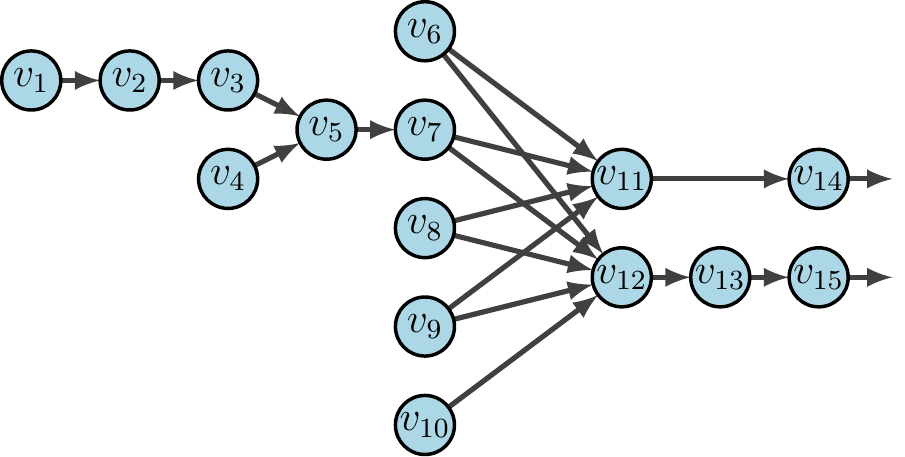}
    \caption{Multi-echelon supply chain network used in simulations containing two distributors $(v_{14},v_{15})$ and $13$ suppliers $(v_i, 0\leq i\leq 13)$.}
    \label{fig:simulation_network}
\end{figure}

\subsection{Simulation Implementation}

In the simulations, we include the 
production policy introduced in \eqref{eq:init_inv_dynamics} jointly with the optimal request developed in \eqref{eq:threshold_policy}. Also, we consider the scenario where the requests of each firm are not always what they are supplied back. To that end, instead of using the amount ordered, $N_i^k B_{ir}$, as the quantity of goods flowing between firms, we define a separate variable $\Tilde{u}_{ij}^k$, which represents
the amount of product $z_j$ delivered to $v_i$ at time $k$. 
As stated above, these two quantities $\Tilde{u}_{ij}^k$ and $N_i^k B_{ir}$ may not be equal.
Let $\Tilde{y}_{ir}^k = x_{ir}^k + \Tilde{u}_{ij}^k$ be the new variable for the on-hand stock of input good $z_r$ at firm $v_i$.
When the demand for a product $z_i$ exceeds the sum of the inventory and the production capacity 
(i.e., $\omega_i^k > N_i^k + x_{ii}^k$), following~ \eqref{eq:init_inv_dynamics}, 
the inventory of product $z_i$ would become negative.
This outcome does not make physical sense, and thus we 
modify~\eqref{eq:init_inv_dynamics} to become 
\begin{equation}\label{eq:inv_dynamics}
    x_{ir}^{k+1} = \begin{cases} 
      \max \left(0,x_{ii}^k + \pi(y_i^k) - \omega_i^k\right)  & r = i \\
    \Tilde{y}_{ir}^k - \pi(y_i^k)B_{ir} & r \neq i .
      \end{cases}
\end{equation}
Equation~\eqref{eq:inv_dynamics} is interpreted as follows: firm $v_i$ will ship out exactly as much good $z_i$ as is demanded, up to the maximum amount that they physically can.
We assume any demand not satisfied at a given time-step is lost to the firms in the considered network, as there is some other competitor who is able to absorb that excess demand. 
Additionally, in these situations each firm requires one
input good from each of the connected upstream firms for manufacturing their output good (i.e., each $B_i$ is a binary vector).

Each firm has a set of costs/penalties that subtract from profits:  
the production cost per unit $c_i$, the penalty for having a shortage $s_i$, and the holding cost $h_i$.
The values selected for the implementations in this section were randomly chosen, constrained by
\begin{align}
    &\sum_{j\in \mathcal{N}_i^+}s_j \leq c_i \leq s_i \label{eq:cost_constraint} \\
    &s_i \geq s_j,\ \ \forall j\in\mathcal{N}_i^+ \label{eq:shortage_constraint}.
\end{align}
Equation~\eqref{eq:cost_constraint} states that the cost to produce a good at firm $v_i$ should be between the sale price of $s_i$ and the purchase price of all incoming goods.
Recall that in these simulations, only one of each input good is used to manufacture the output.
Equation~\eqref{eq:shortage_constraint} states that the sale price of a good should be greater than or equal to the sale price of each input good, or that as goods are processed more and more, their value increases.

For these simulations, the production function used is 
\begin{equation}
    N_i^k = \min\left\lfloor \Tilde{y}_i^k \oslash B_i \right\rfloor,
\end{equation}
where $\Tilde y_i^k = \left[\Tilde y_{i1}^k \ \Tilde{y}_{i2}^k \ \cdots \  \Tilde y_{im}^k\right]^T$, $\oslash$ is the element-wise division operator, and $\left\lfloor\cdot\right\rfloor$ is the floor function.
The amount of product made at time $k$ will be the maximum possible using the on-hand stock of each product $\Tilde y_{ir}^k$, which is limited by the minimum whole number of products that could be created from each input good.

The initial state of each inventory in the network is a uniformly randomly sampled value below the maximum inventory levels,~$x_{ir}^0\in\left[0,x_{ir}^U\right]$.
In order to use 
\eqref{eq:F},\eqref{eq:threshold_definition} and \eqref{eq:threshold_policy} we must first know the demand distribution seen by each firm. 
Thus, we use a set of Monte Carlo simulations to map the demand from the suppliers' exogenous demand
backwards to the upstream firms, and approximate the distribution of the demand per stage at each supplier firm.
The expected demand distributions used by the Monte Carlo simulations for each distributor firm $v_{14}$ and $v_{15}$ are $\mathcal{N}(32,4)$.

\subsection{Ideal Simulations}

The first simulations followed all the assumptions made in Sections~\ref{sec:model} and~\ref{sec:optimization}. 
Plots for a sample of firms from one such simulation can be seen in Figure~\ref{fig:simulation_ideal}.
Notice that under ideal conditions, the output good inventory level never falls below the inventory policy minimum, which is the desired behavior.



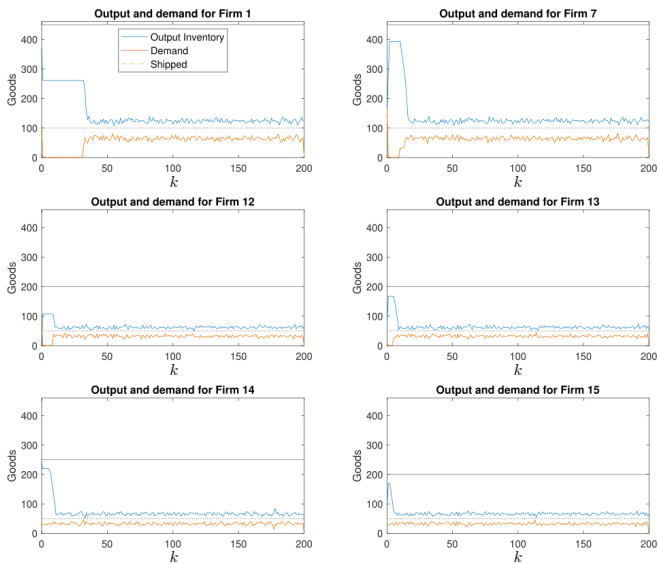
\begin{figure}
    \centering
    \begin{overpic}[width=\columnwidth]{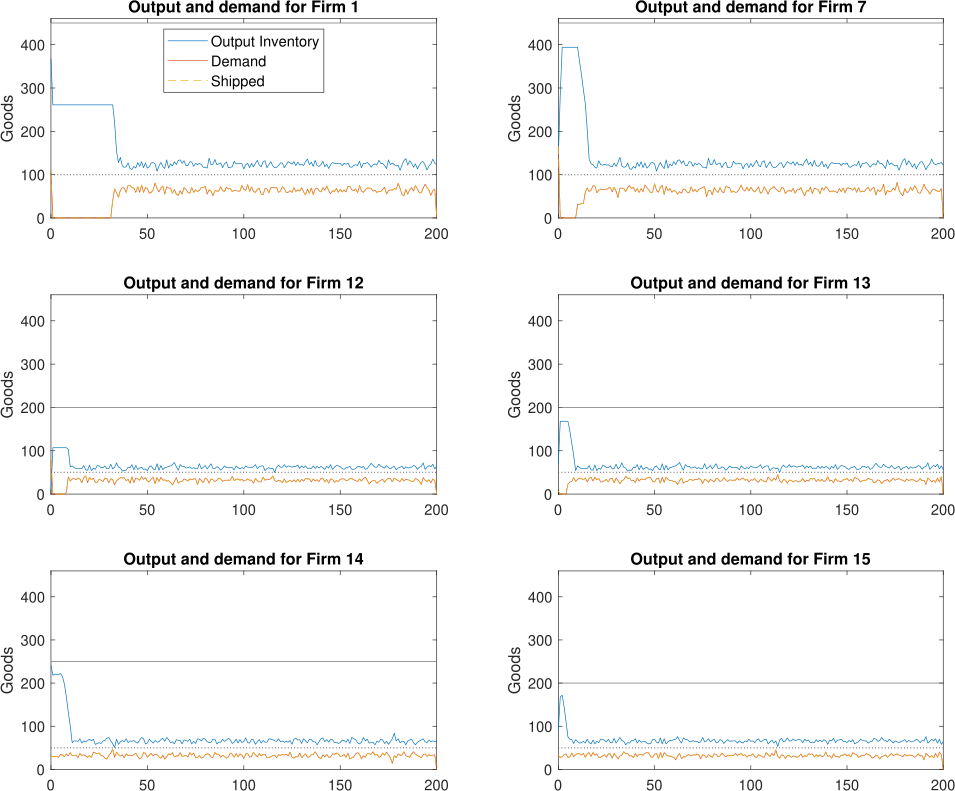}
    \put(25,-2.3){\scalebox{0.6}{$k$}}%
    \put(78,-2.3){\scalebox{0.6}{$k$}}%
    \put(25,26.5){\scalebox{0.6}{$k$}}%
    \put(78,26.5){\scalebox{0.6}{$k$}}%
    \put(25,55.5){\scalebox{0.6}{$k$}}%
    \put(78,55.5){\scalebox{0.6}{$k$}}%
    \end{overpic}
    \caption{Implementation of Section~\ref{sec:optimization} with demand distributions of $\mathcal{N}(32,4)$ for each distributor.}
    \label{fig:simulation_ideal}
\end{figure}

\subsection{Supplier Outages}
The first non-ideal case we simulate to check the 
robustness of \eqref{eq:threshold_policy} on the graph in Figure~\ref{fig:simulation_network} is when supplier $v_1$ has a total halt in production for $k\in[100,150]$, and the resulting plots can be see in Figure~\ref{fig:simulation_outage}.
In this scenario, 
several interesting behaviors arise. 
First, even though all orders are made and fulfilled in a single day, it takes time for nodes in further downstream nodes to ``feel" the effect of this outage. 
This delayed impact is because the output inventories of each firm act as a buffer, absorbing the impact until inventory reserves are empty.
The time it takes for a downstream firm to feel the impact of an outage at an upstream firm depends on the minimum inventory policies of every firm in the supply chain 
combined with how much demand there is.

Once the outage is resolved it takes time for each firm to fully recover, as can be seen in Figure~\ref{fig:simulation_outage}. 
Although the outage resolves at time $k=150$, the only firms that return to their ideal inventory levels immediately are the distributors (firms $v_{14}$ and $v_{15}$).
It takes around one additional time-step per echelon upstream from the distributors for the other firms to see ideal inventory levels again, which is because each firm essentially is refilling its buffer before resuming normal ordering behavior.

By these observations, it falls out that the firms producing goods directly for consumers end up experiencing the least impact from far upstream supply shocks. In fact, when the outage only went from $k=100$ to $k=130$, there were only a few time-steps where the distributors were unable to satisfy all their demand.

\subsection{Demand Shocks}
The second type of shocks induced on the network were demand quantities far greater than considered when 
\eqref{eq:threshold_policy} was computed.
The inventory levels of the network when demand was increased each time-step can be seen in Figure~\ref{fig:simulation_demand}.
The first observation to be made from looking at these plots is that the distributors are able to satisfy all the demand coming in until demand begins to exceed some threshold of around 100 units per day.  
Once the demand reaches this point, the quantity of orders filled is essentially constant until demand returns to a value below this threshold. The exact value of this threshold is determined by 
the optimal request being $N_i^{k,\star} = \tilde{M}_i^k$ according to \eqref{eq:threshold_policy}, given that the current inventory level is zero. Under this scenario, even though the threshold $\tilde{M}_i^k$ is the optimal request, it is not able to meet all the demand and 
ensure
\eqref{eq:DP_Algorithm_4} is satisfied.  This 
discrepancy occurs
mainly because at each time-step the interval $\mathcal{I}_i^k$ is increasing in length  
and this behavior
is not being considered in~\eqref{eq:DP_Algorithm_5}.
As can be seen in Figure~\ref{fig:simulation_demand}, 
once the shock ends (i.e., the interval $\mathcal{I}_i^k$ returns to its original 
length) the inventory level at each firm will 
recover and be able to satisfy
\eqref{eq:DP_Algorithm_4}.

When looking upstream from the distributors, notice that it is only the direct suppliers that experience demand that exceeds what they are capable supplying. Any firms further upstream only experience the demand directly communicated to them, essentially being unaffected by the mismatch in demand and production capacity at the end of the supply chain.
This behavior occurs because after enough mappings (i.e., after the demand is fed through the direct suppliers to distributors) the largest orders being made by each firm fall within the bounds of 
$\mathcal{I}_i^k$ for downstream firms.

\subsection{Robustness}
The purpose of the second and third simulation scenarios, where shocks were induced on the network, was to evaluate how an 
ordering policy that is determined to be optimal under ideal conditions developed in~\eqref{eq:threshold_policy} performs under real-world circumstances.
To quantify the relative performance of the network in each scenario, we look at the cumulative sum of penalties across the entire network.
Plots of these costs for each simulation 
can be seen in Figure~\ref{fig:simulation_costs}.
The key takeaway from this figure is that the ordering policy implemented is optimal, but not necessarily robust.
The policy could be made more robust by accounting for some probability of a shock occurring at each time-step, however that would result in 
higher costs for the 
 ideal scenario, as would be expected, illustrating the 
tradeoff between robustness and optimality.
While we have made this observation, delving into this tradeoff in more depth is outside the scope of this work and should be considered as  future work. 

\begin{figure}
    \centering
    \begin{overpic}[width=\columnwidth]{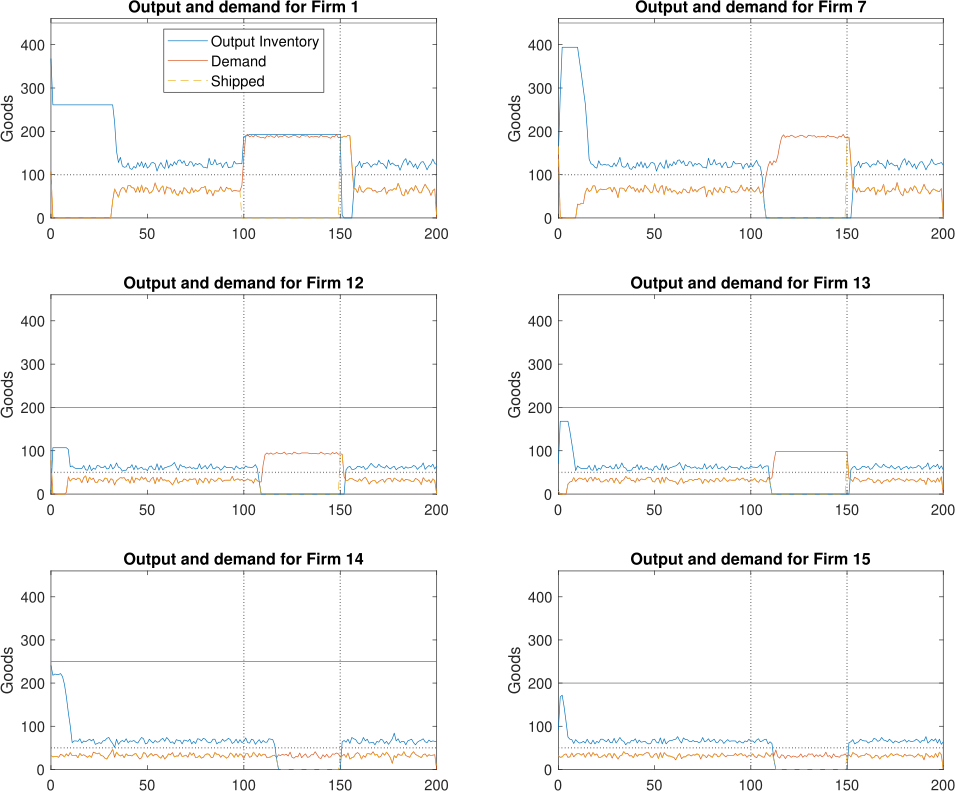}
    \put(25,-2.3){\scalebox{0.6}{$k$}}%
    \put(78,-2.3){\scalebox{0.6}{$k$}}%
    \put(25,26.5){\scalebox{0.6}{$k$}}%
    \put(78,26.5){\scalebox{0.6}{$k$}}%
    \put(25,55.5){\scalebox{0.6}{$k$}}%
    \put(78,55.5){\scalebox{0.6}{$k$}}%
    \end{overpic}
    \caption{Simulations of Figure~\ref{fig:simulation_network} with demand distributions of $\mathcal{N}(32,4)$ for each distributor, where supplier $v_1$ is unable to ship any of their output good from $k=100$ to $k=150$, which are marked by the vertical lines.}
    \label{fig:simulation_outage}
\end{figure}


\begin{figure}
    \centering
    \begin{overpic}[width=\columnwidth]{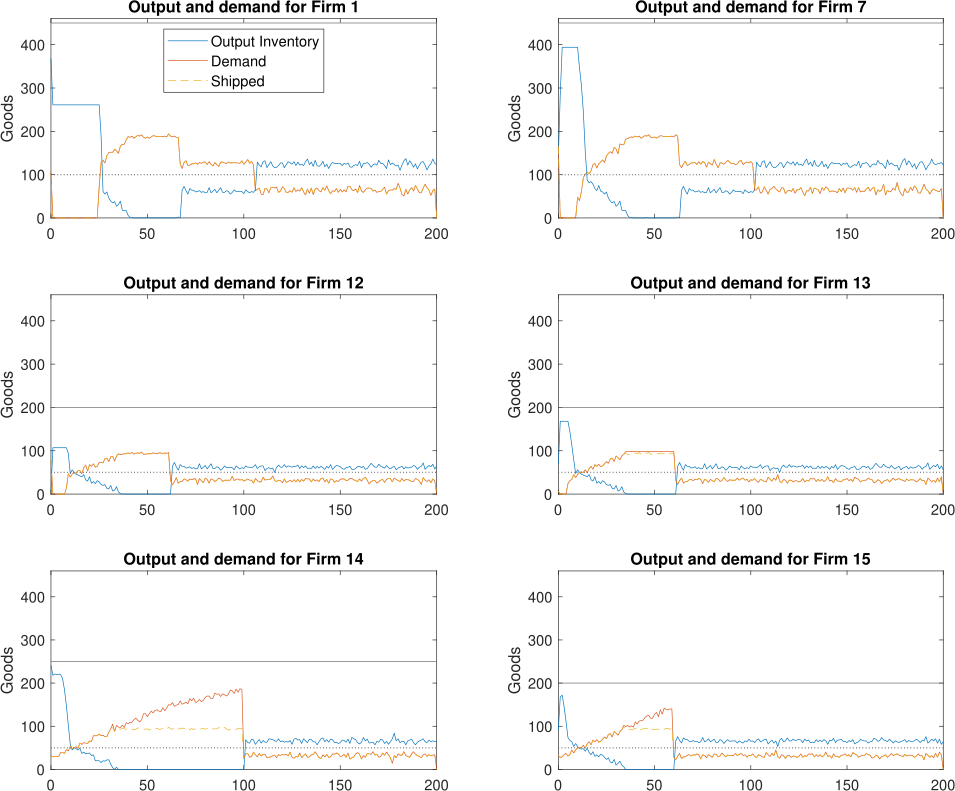}
    \put(25,-2.3){\scalebox{0.6}{$k$}}%
    \put(78,-2.3){\scalebox{0.6}{$k$}}%
    \put(25,26.5){\scalebox{0.6}{$k$}}%
    \put(78,26.5){\scalebox{0.6}{$k$}}%
    \put(25,55.5){\scalebox{0.6}{$k$}}%
    \put(78,55.5){\scalebox{0.6}{$k$}}%
    \end{overpic}
    \caption{Simulations on Figure~\ref{fig:simulation_network} with ideal demand distributions of $\mathcal{N}(32,4)$ for each distributor firm, where the demand mean increases by $2$ each time-step over $k\in[0,60]$ for both distributors $v_{14},v_{15}$, and then increases by $1$ over $k\in[61,100]$ for distributor $v_{14}$. 
    The demands shift back to the expected $\mathcal{N}(32,4)$ after these time-spans.
    }
    \label{fig:simulation_demand}
\end{figure}

\begin{figure}
    \centering
    \begin{overpic}[width=.5\columnwidth]{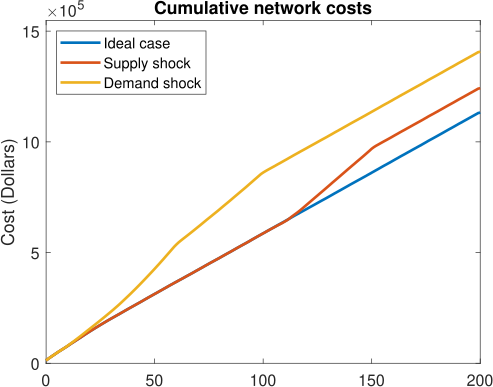}
    \put(52,-4.5){\scalebox{0.6}{$k$}}%
    \end{overpic}
    \caption{The cumulative aggregate costs for the entire network, for each of the three simulations plotted previously.
    }
    \label{fig:simulation_costs}
\end{figure}

\section{Conclusion} \label{sec:conclusion}

In this paper, we formulate the stochastic control problem of a multistage inventory management in a supply network. The structure of the problem is studied rigorously based on the development of a sequence of mapping functions that turn out to be critical to obtain a closed-form solution for the optimal request and cost-to-go function at each firm. In order to understand how our optimal request policies behave in a supply chain network we incorporate rudimentary production decisions into each firm. 

Through our simulations, we demonstrate that the policy developed herein performs as expected in ideal scenarios, where the conditions assumed in the analysis for the optimal requests hold. 
Further, we demonstrate that 
when there are internal shocks or shifts in exogenous demand on the system the whole network is able to recover its usual dynamics based on the policy proposed. We also analyze the impact on firms depending on their echelon in the SN and conclude that the production policy jointly with the optimal request provides a natural ``buffer" to each firm.
%
In future work we will develop analytical and numerical results for a joint production-inventory optimal policy, and investigate the possibility of including forecast updates on the production costs and the probability of supplier shocks to improve the robustness of the scheme proposed.

\bibliographystyle{IEEEtran}
\bibliography{refs}

\appendix
\begin{proposition}\label{prop:convex_general}
Let $X$ be a nonempty set with $S_x$ a nonempty set for each $x \in X$. Let $C = \{(x,y):y\in S_x, x\in X\}$, let $V$ be a real-valued function on $C$ and define
\begin{align*}
    f(x) = \text{inf}\{J(x,y):y\in S_x\}, x\in X.
\end{align*}
If $C$ is a convex set and $J$ is a convex function on $C$, then $f$ is a convex function on any convex subset of $X$ where $f > -\infty$.
\end{proposition}
\begin{proof}
Let $x_1$ and $x_2$ such that $f(x_1)>-\infty$, $f(x_2)>-\infty$. Pick $\alpha \in (0,1)$, and let $x = \alpha x_1 + (1-\alpha)x_2$. Now, by definition of the function $f$, we have that:
\begin{align*}
    f(\alpha x_1 + (1-\alpha)x_2) &= \underset{\alpha}{\inf}\big\{J(\alpha x_1 + (1-\alpha)x_2,y)\big\} \\
    &\leq \underset{\alpha}{\inf}\big\{\alpha J(x_1,y)+(1-\alpha)J(x_2,y)\big\} \\
    &\leq \underset{\alpha}{\inf}\big\{ J(x_1,y) \big\} + (1-\alpha)\inf\big\{J(x_2,y)\big\} \\
    \normalsize
    &= \alpha f(x_1) + (1-\alpha) f(x_2),
\end{align*}
with the first inequality due to the convexity of $J$ on $C$. The final equality establishes convexity of $f$.\\
\end{proof}
\textbf{Proof of Lemma 2:}
\\We will prove all the claims stated in Lemma \ref{lem:T_G}  by constructing the backward recursion.\\ 
\textbf{First iteration}: For stage $T-1$, $G_i^{T-1}(y_{ii})$ is convex since $J_i^{T-1}$ is convex based on  Proposition \ref{prop:convex_a} and it is differentiable such that
\begin{align*}
    {G_i^{T-1}}^\prime(y_{ii}) &= \pdv{\mathbb{E}[f_i(y_{ii}-\omega_i^{T-1})]}{y_{ii}} = \mathbb{E}[f_i^\prime(y_{ii}-\omega_i^{T-1})].
\end{align*}
Thus, we have that $(G_i^{T-1})^\prime = F_i^{T-1}$, where 
\begin{align*}
    F_i^{T-1}(y_{ii})&=-s_i\mathbb{P}(\omega_i^{T-1}>y_{ii}) + h_i\mathbb{P}(\omega_i^{T-1}\leq y_{ii})\\
    &= -s_i + (s_i + h_i)\mathbb{P}(\omega_i^{T-1}\leq y_{ii}).
\end{align*}
Note that $F_i^{T-1}(y_{ii})$ is nondecreasing with $F_i^{T-1,\text{max}}$ and $F_i^{T-1,\text{min}}$ defined as in Lemma \ref{lem:T_G}. Based on the intermediate value theorem, we know that $M_i^{T-1}$ exists and is well-defined whenever $F_i^{T-1,\text{min}}<-c_i<F_i^{T-1,\text{max}}$. Now, consider the optimization at stage $T-1$ 
\begin{align}\label{eq:J_T_1_app}
    -c_ix_{ii}^{T-1}+\underset{y_{ii}^{T-1}\geq x_{ii}^{T-1}}{\text{min}}\big\{c_iy_{ii}^{T-1}+\mathbb{E}\big[f\big(y_{ii}^{T-1}-\omega_i^{T-1}\big)\big]\big\}
\end{align}
in which the objective function is convex and its derivative with respect to $y_{ii}^{T-1}$ is $c_i+F_i^{T-1}(y_{ii})$. If $F_i^{T-1,\text{min}}>-c_i$, $c_i+F_i^{T-1}(y_{ii})$ is nondecreasing $\forall y_{ii} \in [\underline{a}_{ii},\overline{a}_{ii}]$ ; if $F_i^{T-1,\text{max}}<-c_i$, $c_i+F_i^{T-1}(y_{ii})$ is nonincreasing $\forall y_{ii} \in [\underline{a}_{ii},\overline{a}_{ii}]$; if $F_i^{T-1,\text{min}}<-c_i<F_i^{T-1,\text{max}}$, by the first-order necessary condition for a point to be optimal, we have that $\big[F_i^{T-1}\big]^{-1}(-c_i)$ is the unconstrained minimizer of \eqref{eq:J_T_1_app}. Since $y_{ii}^k$ is subjected to \eqref{eq:constraint_on-hand}, we can say that the feasible range of the mapping function will be in the interval $\big[F_i^{T-1}(x_{ii}^L+D_i^k),F_i^{T-1}(x_{ii}^U)\big]$ since $F_i^{T-1}$ is nondecreasing. Thus, we can summarize the unconstrained minimizer of \eqref{eq:J_T_1_app} in the form of \eqref{eq:threshold_definition}. By using Lemma \ref{lem:dp} and $y_{ii}^{T-1}=x_{ii}^{T-1}+N_i^{T-1}$ , the optimal request at stage $T-1$ is of the form $N_i^{T-1,\star}(x_{ii})=(\tilde{M}_i^{T-1}-x_{ii}^k)_+$ as stated in Theorem \ref{theo:J}. Then depending on whether $\tilde{M}_i^{T-1}-x_{ii}^k \geq 0$, $J_i^{T-1}$ has the form
\begin{multline}\label{eq:J_T-1_proof}
    J_i^{T-1}(y_{ii}^{T-1})=\\
\begin{cases}
c_i\big(\tilde{M}_i^{T-1}-x_{ii}^k\big)+G_i^{T-1}\big(\tilde{M}_i^{T-1}\big) & \text{if $\tilde{M}_i^{T-1}\geq x_{ii}^{T-1}$},\\
 G_i^{T-1}\big(x_{ii}^{T-1}\big) & \text{if $\tilde{M}_i^{T-1}<x_{ii}^{T-1}$.}
\end{cases}
\end{multline}
Note that based on the monotonicity of the function $F_i^{T-1}$ and the threshold $\tilde{M}_i^{T-1}$ previously defined, we have that:
\begin{align*}
    &\big\{y_{ii}-\omega_i^{T-2} \leq \tilde{M}_i^{T-1}\big\} \\
    &= \big\{y_{ii}-\omega_i^{T-2}\leq \big(T_i^{T-1}\big)^{-1}(-c_i)\big\}~\cup~\{-c_i<T_i^{T-1,\min}\}\\
    &= \{T_i^{T-1}(y_{ii}-\omega_i^{T-2})\leq -c_i\}~\cup~\{-c_i<T_i^{T-1,\min}\}\\
    &= \{T_i^{T-1}(y_{ii}-\omega_i^{T-2})\leq -c_i\}.
\end{align*}
From this last equation we can conclude that $y_{ii}$ in the feasible set $[\underline{a}_{ii},\overline{a}_{ii}]$,~$y_{ii}-\omega_i^{T-2}\leq\tilde{M}_i^{T-1}$ if and only if $T_i^{T-1}(y_{ii}-\omega_i^{k-1})\leq-c_i$.\\
\textbf{$\boldsymbol{k}$th Iteration:} Assume that all the claims hold for stage $k+1$. That is $J_i^{k+1}(y_{ii})$ has the form as in \eqref{eq:DP_Algorithm_2} and $\mathds{1}\big(T_i^{k+1}(y_{ii}-\omega_i^{k})\leq -c_i\big)=\mathds{1}\big(x_{ii}^{k+1}\leq\tilde{M}_i^{k+1}\big)$, we can check by definition of $G_i^k$ in \eqref{eq:G} that, 
\begin{equation}\label{eq:G_t_proof}
    G_i^k(y_{ii}) = \mathbb{E}\big[f_i(y_{ii}-\omega_i^k)+J_i^{k+1}(y_{ii}-\omega_i^k)\big]
\end{equation}
and it is convex by Proposition \ref{prop:convex_a}. Now, by computing the derivative of $G_i^{k+1}$, such that ${G_i^{k+1}}^\prime(y_{ii}) = T_i^{k+1}(y_{ii})$, we obtain
\begin{align*}
    {G_i^k }^\prime (y_{ii}) &= \mathbb{E} \bigg\{ f^\prime(y_{ii}-\omega_i^k)-c_i \mathds{1}\big(T_i^{k+1}(y_{ii}-\omega_i^k)\leq -c_i\big) \\
    &\quad +T_i^{k+1}(y_{ii}-\omega_i^k)\mathds{1}\big(T_i^{k+1}(y_{ii}-\omega_i^k)>-c_i\big)\bigg\},
\end{align*}
that is, ${G_i^k }^\prime=T_i^k$ and $T_i^k$ is nondecreasing. The rest of the arguments in Lemma \ref{lem:T_G} follow from the same arguments as in the base case.\\  
\\\textbf{Proof of Theorem \ref{theo:J}:}
\\By induction. For the base case, we have that the cost-to-go function is the one obtained in \eqref{eq:J_T-1_proof}. Then, consider the stage-$k$ optimization
\begin{align*}
    &\underset{N_i^k\geq 0}{\text{min}}\big\{c_iN_i^k+\mathbb{E}\big[f_i(x_{ii}^{k}+N_i^k-\omega_i^{k})+J_i^{k+1}(x_{ii}^k+N_i^k-\omega_i^k)\big]\big\} \\
    &= -c_ix_{ii}^{k}+\underset{y_{ii}^{k}\geq x_{ii}^{k}}{\text{min}}\big\{c_iy_{ii}^{k}+G_i^k(y_{ii}^k)\},
\end{align*}
where we follow the usual change of variable and \eqref{eq:G_t_proof}. Repeating the arguments as in the base case and following the proof of Lemma \eqref{lem:T_G} with $G_i^k(y_{ii}^k)$ in place of $\mathbb{E}[f_i(y_{ii}-\omega_i^k)] = G_i^{T-1}(y_{ii}^{T-1})$ concludes the proof.
\end{document}


\begin{tikzpicture}

\definecolor{color1}{rgb}{1,0.498039215686275,0.0549019607843137}
\definecolor{color0}{rgb}{0.12156862745098,0.466666666666667,0.705882352941177}

\begin{axis}[view={0}{90},
legend cell align={left},
legend style={fill opacity=0.8, draw opacity=1, text opacity=1, draw=white!80.0!black},
tick align=outside,
tick pos=both,
x grid style={lightgray!92.02614379084967!black},
xlabel={Time (k)},
xmin=0, xmax=30,
xtick style={color=black},
y grid style={lightgray!92.02614379084967!black},
ylabel={Good $z_2$},
ymin=0, ymax=91,
ytick style={color=black},
grid=both,
major grid style={line width=.2pt,draw=gray!50}
]
\addplot [line width = 1.8pt, color0]
table {%
0  33
1  72
2  66
3  60
4  53
5  49
6  47
7  50
8  49
9  53
10 48
11 51
12 47
13 51
14 54
15 48
16 49
17 49
18 53
19 54
20 49
21 48
22 52
23 51
24 52
25 54
26 49
27 53
28 51
29 52
30 54
};
\addlegendentry{Inventory}
\addplot [line width = 1.8pt, dashed, color1]
table {%
0  14
1  6
2  6
3  7
4  7
5  9
6  6
7  7
8  3
9  8
10 5
11 9
12 5
13 2
14 8
15 7
16 7
17 3
18 2
19 7
20 8
21 4
22 5
23 4
24 2
25 7
26 3
27 5
28 4
29 2
30 0
};
\addlegendentry{Demand}

\end{axis}

\end{tikzpicture}